\documentclass{amsart}
\usepackage{amssymb,amsmath,amsthm,hyperref}

\usepackage[normalem]{ulem}

\theoremstyle{theorem}
\newtheorem{theorem}{Theorem}

\newtheorem{lemma}[theorem]{Lemma}

\theoremstyle{definition}

\allowdisplaybreaks

\DeclareMathOperator{\ck}{crank}

\title[Extending a Refinement of the Crank--Mex Theorem]{Extending Andrews and Newman's Refinement of the Crank--Mex Theorem}

\author{George Andrews and Brian Hopkins}
\address{}
\email{}
\date{}

\begin{document}

\begin{abstract}
The crank--mex theorem states that the number of integer partitions of $n$ with nonnegative crank equals the number with odd minimal excludant (mex).  Andrews and M.\ Newman recently refined that result in terms of the number of parts greater than one.  Here, we establish and expand a complementary result connecting the partitions with even mex, having fixed points, with negative crank, and with positive crank, all refined in terms of number of parts greater than one.  We provide both analytic and combinatorial proofs.
\end{abstract}

\maketitle

\section{Introduction}

A partition of a positive integer $n$ is $\lambda = (\lambda_1, \ldots, \lambda_t)$ where each part $\lambda_i$ is a positive integer, the sum of the parts is $n$, and the parts are ordered such that $\lambda_1 \ge \ldots \ge \lambda_t$.  Write $P(n)$ for the set of partition of $n$ and $p(n) = |P(n)|$; we will consistently use lowercase letters for the sizes of sets.  By convention, $p(0) = 1$.

An important partition statistic for our project is the Durfee square: Given a partition $\lambda$, the size of its Durfee square is the greatest index $i$ for which $\lambda_i \ge i$.  See Andrews and Eriksson \cite{ae} for more background on partitions.

We will also use $q$-series notation such as 
\begin{gather*}
(a;q)_n = \prod_{i = 0}^{n-1} (1-aq^i) \quad \text{for $0 \le n \le \infty$},\\
{a \brack b} = \frac{(q;q)_a}{(q;q)_b (q;q)_{a-b}} \quad \text{if $0 \le b \le a$, else $0$}, \\
{}_2\phi_1\binom{a, b; q, z}{c} = \sum_{n = 0}^\infty \frac{(a;q)_n (b;q)_n z^n}{(q;q)_n (c;q)_n};
\end{gather*}
see Andrews \cite{a}.

Our results connect several partition statistics including, most importantly, the crank.  This was named by Dyson in 1944 in order to witness the modulo 11 Ramanujan congruence but not found until 1988 by Andrews and Garvan \cite{ag}.  For a partition $\lambda$, let 
$\omega(\lambda)$ be the number of parts 1 and $\mu(\lambda)$ the number of parts greater than $\omega(\lambda)$. The crank of $\lambda$ is then
\[ \ck(\lambda) =\begin{cases} \lambda_1 & \text{ if $\omega(\lambda)=0$},\\
\mu(\lambda) - \omega(\lambda) & \text{ if $\omega(\lambda)>0$}.
\end{cases} \]
We consider partitions with several ranges of cranks, such as $M_{\ge 0}(n)$ for nonnegative crank.

The mex of a partition $\lambda$ is the smallest positive integer that is not a part; see Andrews--D. Newman 2019 \cite{an19}.  Let $X_o(n)$ be the partitions of $n$ with odd mex, $X_e(n)$ for even mex.

In 2020, both Andrews--D.\ Neman \cite{an20} and Hopkins--Sellers \cite{hs20} established
\begin{equation}
x_o(n) = m_{\ge 0}(n) \label{cm}
\end{equation}
which has come to be known as the crank--mex theorem.  In 2023, Konan \cite{k} gave a combinatorial proof of this result.

For the refinements, we introduce another partition statistic:\ Let $\beta(\lambda)$ be the number of parts of $\lambda$ that are greater than one.  The length of $\lambda =  (\lambda_1, \ldots, \lambda_t)$ is $t$, the number of parts, so that $\beta(\lambda) = t - \omega(\lambda)$.

We further refine each set of partitions with the parameter $k$ indicating the number of parts greater than one.  For instance, $M_{<0}(n,k)$ is the set of partitions $\lambda$ of $n$ with $\ck(\lambda) < 0$ and $\beta(\lambda) = k$.

Andrews and M.\ Newman recently established a refinement of the crank--mex theorem \eqref{cm}, proving analytically that
\[x_o(n,k) = m_{\ge 0}(n,k)\]
and requesting a combinatorial proof.

The current work incorporates another partition statistic.  In 2022, Blecher and Knopfmacher \cite{bk} applied the notion of fixed points to partitions:\ A partition $\lambda = (\lambda_1, \ldots, \lambda_t)$ has a fixed point if $\lambda_i = i$ for some index $i$.  Since parts are listed in nonincreasing order, a partition can have at most one fixed point.  Let $F(n)$ be the partitions of $n$ with a fixed point.  Note that a partition with a fixed point $i$ has Durfee square size $i$.  A partition with Durfee square size $d$ without a fixed point has $\lambda_d > d$.

In 2024, Hopkins and Sellers \cite{hs24} connected many of the ideas mentioned so far by proving that
\[ x_e(n) = f(n) = m_{< 0}(n) = m_{> 0}(n).\]

Our main result is the following refinement of this result.

\begin{theorem} \label{big}
Given positive integers $n$ and $k$, the following are equinumerous:
\begin{itemize}
\item[(a)] partitions of $n$ with even mex and $k$ parts greater than 1,
\item[(b)] partitions of $n$ with a fixed point and $k+1$ parts greater than 1,
\item[(c)] partitions of $n$ with negative crank and $k$ parts greater than 1,
\item[(d)] partitions of $n$ with positive crank and $k+1$ parts greater than 1.
\end{itemize}
That is, $x_e(n,k) = f(n,k+1) = m_{<0}(n,k) = m_{>0}(n,k+1)$.
\end{theorem}

Note that $X_e(n,0)$, $F(n,0)$, and $M_{<0}(n,0)$ each consist of the single partition $(1^n)$ which corresponds to $M_{>0}(n,1) = (n)$ so that $x_e(n,0) = f(n,0) = m_{<0}(n,0) = m_{>0}(n,1)$.  In other words, the equality for (a), (c), and (d) extends to the case $k = 0$.  The minor exception around $F(n,0)$ will arise in both proofs.  We use the notation $F^*(n,k+1)$ to include this special consideration for the $k=0$ case.

In the next section, we give an analytic proof of Theorem \ref{big}.  Section 3 presents three bijections (one from Konan) that provide a combinatorial proof of the theorem.  We conclude with comments on Andrews and Newman's request for a combinatorial proof of their result.

\section{Generating function proof}

We determine generating functions for each type of partition where $q$, as usual, contributes to the partition sum and $z$ tracks the parts greater than one.  We will also use the following crank generating function where $y$ keeps track of the crank:
\begin{equation}
(1-q) + \sum_{n \ge 1} \frac{q^n y^n}{(q^2;q)_{n-1}} + \sum_{n \ge 1} \frac{q^n y^{-n}}{(q^2;q)_{n-1}} \sum_{m \ge 0} \frac{q^{m(n+1)} y^m}{(q;q)_m}. \label{cgf}
\end{equation}
This follows from the crank discovery in 1988 \cite{ag,g}; see \cite[\S2]{an25} for details.

(a) The generating function for partitions with even mex is
\[ \sum_{n \ge 1}q^{1+\cdots+(2n-1)} \prod_{\substack{j = 1 \\ j \ne 2n}}^\infty \frac{1}{1-q^j} \]
as the exponent $1 + \cdots + (2n-1)$ guarantees that at least one of each part $1, \ldots, 2n-1$ occurs and the product allows for any additional part other than $2n$ \cite{an20,hs20}.

To track the number of parts greater than one, we introduce $z$:
\[ \sum_{n \ge 1} \sum_{z \ge 0} x_e(n,k) q^n z^k = \sum_{n \ge 1}z^{2n-2} q^{1+\cdots+(2n-1)} \frac{1}{1-q} \prod_{\substack{j = 2 \\ j \ne 2n}}^\infty \frac{1}{1-zq^j} \]
where $z^{2n-2}$ accounts for the required parts $2, \ldots, 2n-1$ and the product has been split with $zq^j$ only in the geometric series for $j \ge 2$.

Continuing,
\begin{align*}
\sum_{n \ge 1} \sum_{z \ge 0} x_e(n,k) q^n z^k & = \sum_{n \ge 1}z^{2n-2} q^{1+\cdots+(2n-1)} \frac{1}{1-q} \prod_{\substack{j = 2 \\ j \ne 2n}}^\infty \frac{1}{1-zq^j} \\
& = \frac{1}{(1-q)(zq^2;q)_\infty} \sum_{n \ge 1} z^{2n-2} q^{\binom{2n}{2}} (1-z q^{2n}) \\
& = \frac{1}{(1-q)(zq^2;q)_\infty} \sum_{n \ge 1} \left(z^{2n-2} q^{\binom{2n}{2}} - z^{2n-1} q^{\binom{2n+1}{2}} \right) \\
& = \frac{1}{(1-q)(zq^2;q)_\infty} \sum_{n \ge 0} (-1)^n z^{n-2} q^{\binom{n}{2}} \\
& = \frac{1}{(1-q)(zq^2;q)_\infty} \sum_{n \ge 2} (-1)^n z^{n-2} q^{\binom{n}{2}} \\
& = \frac{1}{(1-q)(zq^2;q)_\infty} \sum_{n \ge 1} (-1)^{n-1} z^{n-1} q^{\binom{n+1}{2}}
\end{align*}
which we call $E(z,q)$.

(b) The generating function for partitions with a fixed point is
\[ \sum_{n \ge 1} \frac{q^{n^2}}{(q;q)_{n-1} (q;q)_n}\]
where $n^2$ accounts for the Durfee square, $(q;q)_n$ for the partition $\beta$ below the Durfee square with first part at most $n$, and $(q;q)_{n-1}$ for the partition $\alpha'$ to the right of the Durfee square with first part at most $n-1$ to insure that the $n$th part of the overall partition is $n$, i.e., a fixed point \cite{hs24}.

To track the number of parts greater than one, we introduce $z$:
\[ \sum_{n \ge 1} \sum_{z \ge 0} f(n,k) q^n z^k = \sum_{n \ge 1} \frac{z^n q^{n^2}}{(q;q)_{n-1} (1-q)(zq^2;q)_{n-1}} \]
where $z^n$ accounts for the first $n$ parts (the Durfee square and $\alpha'$ to its right) and $(q;q)_n$ has been factored with $z$ applied to the parts greater than one in $\beta$.  Note that this interpretation is valid only for partitions with Durfee square of size at least two; this excludes the partitions with all parts 1 mentioned after the Theorem \ref{big} statement.

Continuing,
\begin{align*}
\sum_{n \ge 1} \sum_{z \ge 0} f(n,k) q^n z^k & = \sum_{n \ge 1} \frac{z^n q^{n^2}}{(q;q)_{n-1} (1-q)(zq^2;q)_{n-1}}  \\
& = \frac{zq}{1-q} \sum_{n \ge 0} \frac{z^n q^{n^2+2n}}{(q;q)_n (1-q)(zq^2;q)_n}  \\
& = \frac{zq}{1-q} \lim_{\tau \to 0} {}_2\phi_1\binom{q/\tau, q/\tau; q, zq\tau^2}{zq^2} \\
& = \frac{zq}{1-q} \cdot \frac{1}{(zq^2;q)_\infty} \lim_{\tau \to 0} {}_2\phi_1\binom{q, q/\tau; q, zq\tau}{zq^2\tau} \\
& \qquad \text{[by the second Heine transform]} \\
& = \frac{zq}{(1-q)(zq^2;q)_\infty} \sum_{n \ge 0} (-z)^n q^{\binom{n+1}{2}+n} \\
& = \frac{z}{(1-q)(zq^2;q)_\infty} \sum_{n \ge 0} (-z)^n q^{\binom{n+2}{2}} \\
& = \frac{z}{(1-q)(zq^2;q)_\infty} \sum_{n \ge 1} (-1)^{n-1} z^{n-1} q^{\binom{n+1}{2}} \\
& = z E(z,q).
\end{align*}

(c) For negative crank, we exclude all terms in the crank generating function \eqref{cgf} where the exponent on $y$ is nonnegative, leaving
\[ \sum_{n \ge 0} \frac{q^n y^{-n}}{(q^2;q)_{n-1}} \sum_{m =0 }^{n-1} \frac{q^{m(n+1)} y^m}{(q;q)_m}. \]
Next, we set $y=1$ and introduce $z$ to track the parts greater than one:
\begin{align*}
 \sum_{n \ge 1} \sum_{z \ge 0} m_{<0}(n,k) q^n z^k & = \sum_{n = 0}^\infty \frac{q^n}{(zq^2;q)_{n-1}} \sum_{m =0 }^{n-1} \frac{z^m q^{m(n+1)}}{(q;q)_m} \\
& = \sum_{m = 0}^\infty \sum_{n = m+1}^\infty \frac{z^m q^{mn+m+n}}{(zq^2;q)_{n-1}(q;q)_m} \\
& = \sum_{m = 0}^\infty \sum_{n = 0}^\infty \frac{z^m q^{mn+m^2+3m+n+1}}{(zq^2;q)_{n-1}(q;q)_m} \\
& = \sum_{n \ge 0} \frac{q^{n+1}}{(zq^2;q)_{n}} \lim_{\tau \to 0} {}_2\phi_1\binom{q/\tau, q/\tau; q, zq^{n+2}\tau^2}{zq^{n+2}} \\
& = \sum_{n \ge 0} \frac{q^{n+1}}{(zq^2;q)_{n}} \cdot \frac{1}{(zq^{n+2};q)_\infty} \sum_{m \ge 0} \frac{(q^2;q)_m (q/\tau;q)_m (zq^{n+1}\tau)^m}{(q;q)_m (z q^{n+3}\tau;q)_m} \\
& \qquad \text{[by the second Heine transform]} \\
& =  \sum_{n \ge 0} \frac{q^{n+1}}{(1-q)(zq^2;q)_\infty} \sum_{m \ge 0} (1-q^{m+1}) (-1)^m q^{\binom{m+1}{2} + m + mn} z^m \\
& =  \sum_{m \ge 0} \frac{(1-q^{m+1})(-z)^m q^{\binom{m+1}{2}}}{(1-q)(zq^2;q)_\infty} \sum_{n \ge 0} q^{(n+1)(m+1)} \\
& =  \frac{1}{(1-q)(zq^2;q)_\infty} \sum_{m \ge 0}  (-z)^m q^{\binom{m+2}{2}} \\
& =  \frac{1}{(1-q)(zq^2;q)_\infty} \sum_{m \ge 1}  (-1)^{m-1} z^{m-1} q^{\binom{m+1}{2}} \\
& = E(z,q).
 \end{align*}
 
 (d) For positive crank, we exclude all terms in the crank generating function \eqref{cgf} where the exponent on $y$ is nonpositive, leaving
 \[ \sum_{n \ge 1} \frac{q^n y^n}{(q^2;q)_{n-1}} + \sum_{n \ge 0} \frac{q^n y^{-n}}{(q^2;q)_{n-1}} \sum_{m \ge n+1} \frac{q^{m(n+1)} y^m}{(q;q)_m} \]
 Next, we set $y=1$ and introduce $z$ to track the parts greater than one.  We want to show
\begin{align}
\sum_{n \ge 1} \sum_{z \ge 0} m_{>0}(n,k) q^n z^k & = zq + \sum_{n \ge 2} \frac{z q^n}{(zq^2;q)_{n-1}} + \sum_{n = 1}^\infty \frac{q^n}{(zq^2;q)_{n-1}} \sum_{m = n+1}^\infty \frac{z^m q^{m(n+1)}}{(q;q)_m} \notag \\
& = \frac{z}{(1-q)(zq^2;q)_\infty} \sum_{n \ge 1} (-1)^{n-1} z^{n-1} q^{\binom{n+1}{2}} \notag \\
& = zE(z,q). \label{dgoal}
\end{align}
Toward that end, 
\begin{align*}
zq + \sum_{n \ge 2} \frac{z q^n}{(zq^2;q)_{n-1}} & - (1-zq) \sum_{m \ge 1} \frac{z^m q^m}{(q;q)_m} \\
& = zq  +(1-zq)  \sum_{n \ge 2}  \frac{z q^n}{(zq;q)_{n-1}} - (1-zq) \left(\frac{1}{(zq;q)_m} - 1 \right) \\
& = zq + (1-zq) \left( \frac{1}{(zq;q)_\infty} - 1 - \frac{zq}{1-zq} \right) - \frac{1-zq}{(zq;q)_\infty} +1 - zq \\
& = 0.
\end{align*}
With this, we can rewrite our desired \eqref{dgoal} as
\[(1-zq) \sum_{n=0}^\infty \sum_{m=n+1}^\infty \frac{z^m q^{mn+m+n}}{(zq;q)_n (q;q)_m} = \frac{z}{(1-q)(zq^2;q)_\infty} \sum_{n \ge 1} (-1)^{n-1} z^{n-1} q^{\binom{n+1}{2}}. \]
We multiply both sides by $(zq^2;q)_\infty$ and compare coefficients of $z^N$.  On the right-hand side, the coefficient is
\[ \frac{(-1)^{N-1} q^{\binom{N+1}{2}}}{1-q}.\]
On the left-hand side, we want the coefficient of $z^N$ in
\[ \sum_{n=0}^\infty \sum_{m=n+1}^\infty \frac{z^m q^{mn+m+n}}{(q;q)_m} (zq^{n+1};q)_\infty\]
which is 
\begin{align*}
\sum_{n=0}^\infty \sum_{m=n+1}^\infty \frac{q^{mn+m+n}}{(q;q)_m} & \cdot \frac{(-1)^{N-m} q^{\binom{N-m+1}{2} + n(N-m)}}{(q;q)_{N-m}} \\
& = \sum_{n=0}^\infty \sum_{m=n+1}^N \frac{1}{(q;q)_N} {N \brack m} (-1)^{N-m} q^{\binom{N-m+1}{2} + nN+m+n} \\
& = \frac{1}{(q;q)_N} \sum_{n=0}^\infty \sum_{m=0}^{N-n-1} {N \brack m} (-1)^m q^{\binom{m}{2} + N+nN+n} \\
& = \frac{1}{(q;q)_N} \sum_{n=0}^\infty q^{N+nN+n} (-1)^{N-n-1} q^{\binom{N-n}{2}} {N-1 \brack N-n-1} \\
& \qquad \text{[by \cite[Lemma 3]{an25}]} \\
& = \frac{1}{(q;q)_N} \sum_{n=0}^{N-1} q^{N+(N-1-n)N + N -n-1} (-1)^n q^{\binom{n+1}{2}} {N-1 \brack n} \\
& = \frac{q^{N^2 + N - 1}}{(q;q)_N} \sum_{n=0}^{N-1} (-1)^n q^{\binom{n}{2}-nN} {N-1 \brack n} \\
& = \frac{q^{N^2 + N - 1}}{(q;q)_N} (q^{-N};q)_{N-1} \\
& \qquad \text{[by the $q$-binomial theorem]} \\
& = \frac{q^{N^2 + N - 1}}{(q;q)_N} (-1)^{N-1} q^{-N(N-1) + \binom{N-1}{2}} (q^2;q)_{N-1} \\
& = \frac{(-1)^{N-1} q^{\binom{N+1}{2}}}{1-q}
\end{align*}
so that the coefficients of $z^N$ match, establishing \eqref{dgoal}.

Together, we have shown that the generating functions for $x_e(n,k)$ and $m_{<0}(n,k)$ are both $E(z,k)$ while the generating functions for $f(n,k)$ (for $k \ge 1$) and $m_{>0}(n,k)$ are both $zE(z,k)$, establishing Theorem \ref{big}.

\section{Combinatorial proof}

\subsection{Bijection between even mex and fixed points}

This bijection connecting (a) and (b) in Theorem \ref{big} follows Konan \cite{k}; we just need to show that it has the desired effect on the number of parts greater than one.  

More specifically, Konan details a generalized bijection that, in one case, connects $X_o(n)$ and the set of partitions of $n$ without a fixed point.  He outlines the changes required for a generalized bijection that, in one case, connects $X_e(n)$ and $F(n)$; that is the bijection we detail.

A few additional definitions are necessary for his bijections.  Let $G(n)$ be the partitions of $n$ without a fixed point, the complement of $F(n)$ in $P(n)$.  Using the terminology of Hopkins and Sellers \cite{hs24}, a partition $\lambda$ has a $j$-fixed point if $\lambda_i = i+j$ for some index $i$.  Let $G_1(n)$ be the partitions of $n$ that have no 1-fixed points.
For a partition $\lambda$, let $d_j(\lambda)$ be the greatest $i$ for which $\lambda_i \ge i+j$; the Durfee square size is $d_0(\lambda)$.  See Hopkins, Sellers, and Yee \cite{hsy} for applications of the related $j$-Durfee rectangle to a generalization of the mex statistic.

Write $\lambda \in X_e(n,k)$ with mex $2j+2$ for $j \ge 0$ as $(2j+1, 2j, \ldots, 2, 1) \cup \kappa$ where $\kappa$ is a partition of $n - j(2j+1)$ without any parts $2j+2$ and having $k - 2j$ parts greater than one (in other words, separate out the parts necessary to make the mex $2j+2$).  Konan describes an iterative map with two cases which repeat until the outcome is $(1) \cup \mu$ for some $\mu \in G_1(n-1)$.

(i) If $\kappa = (\kappa_1, \ldots, \kappa_t)$ has a $(2j+1)$-fixed point, i.e., 
$\kappa_i = i+2j+1$ for some index $i$, then fix $(2j+1, 2j, \ldots, 2, 1)$ and associate $\kappa$ with
\[ (\kappa_1 +1, \ldots, \kappa_{i-1}+1, 2j+2, \kappa_{i+1}, \ldots, \kappa_t). \]
Note that the number of parts greater than one is fixed.

(ii) If $\kappa$ does not have a $(2j+1)$-fixed point, set $d = d_{2j+1}(\kappa)$ and associate $(2j+1, 2j, \ldots, 2, 1) \cup \kappa$ with
\[ (2j-1, \ldots, 1) \cup (\kappa_1 -1, \ldots, \kappa_d - 1, d+2j+1, 2j, \kappa_{d+1}, \ldots, \kappa_t) \]
where, again, the overall number of parts greater than one is fixed.

As Konan explains, the proof that the map $X_o(n)$ to $G(n)$ is well-defined \cite[\S 4]{k} carries over to this map from $X_e(n)$ to $G_1(n-1)$.  

Table \ref{tab:ab} includes the bijection between $X_e(8)$ and $G_1(7)$.  As an example of a partition that requires more than one iteration of the maps, consider $(3,3,2,1) \in X_e(9)$ with mex 4.  Write this as $(3,2,1) \cup (3)$ and apply (ii) to get $(1) \cup (3,3,2)$.  Since the second part 3 is a 1-fixed point, apply (i) to get $(1) \cup(4,2,2)$ where $(4,2,2) \in G_1(8)$.

Konan also provides our final step for this first bijection, giving a bijection between $(1) \times G_1(n-1)$ and $F(n)$ \cite[Lemma 18]{k}.  Given $\lambda = (\lambda_1, \ldots, \lambda_t) \in G_1(n-1)$, set $d = d_1(\lambda)$ and associate $\lambda$ with 
\[(\lambda_1 - 1, \ldots, \lambda_d - 1, d + 1, \lambda_{d+1}, \ldots, \lambda_t) \in F(n)\]
since the $(d+1)$st part is $d+1$.  Note that the number of parts greater than one increases by one with the inclusion of $d+1$ except when $d_1(\lambda) = 0$ which occurs only for $\lambda = (1^{n-1})$ that is associated with $(1^n)$.  This explains the $k = 0$ exception for $F^*(n,k+1)$.

Table \ref{tab:ab} includes the bijection between $(1) \times G_1(7)$ and $F(8)$.  Continuing the previous example of $(3,3,2,1) \in X_e(9)$, the image $(1) \cup (4,2,2)$ with $d_1 = 1$ is sent to $(3,2,2,2) \in F(9)$.  For another example, $(4,4,1) \in X_e(9)$ is sent to $(1) \cup (4,4)$ with $d_1 = 2$ so that the final image is $(3,3,3) \in F(9)$.

We refer the reader to Konan \cite{k} for details of the reverse maps and verifications of these bijections.

\begin{table}[h]
\renewcommand{\arraystretch}{1.1}
\caption{The bijection between $X_e(8,k)$ and $F^*(8,k+1)$ through $(1) \times G_1(7,k)$, with applications of the maps (i) and (ii) included, where values of $k$ are separated by horizontal lines. Recall the special status of $1^8$ where superscripts denote repetition.} \label{tab:ab}
\begin{tabular}{c|c|c}
$X_e(8,k)$ & $(1) \times G_1(7,k)$ & $F^*(8,k+1)$ \\ \hline \hline
$1^8$ & $(1,1^7)$ & $1^8$ \\ \hline
$71$ & $(1,7)$ & $62$ \\
$611$ & $(1,61)$ & $521$\\
$51^3$ & $(1,511)$ & $4211$ \\
$41^4$ & $(1,41^3)$ & $321^3$ \\
$31^5$ & $(1,31^4)$ & $221^4$ \\ \hline
$431$ & $(1,43) \overset{\text{(i)}}{\longrightarrow} (1,52)$ & $422$ \\
$3311$ & $(1,331) \overset{\text{(i)}}{\longrightarrow} (1,421)$ & $3221$  \\
$321^3$ & $(321,11) \overset{\text{(ii)}}{\longrightarrow} (1,3211)$ & $2^311$ \\ \hline
$3221$ & $(321,2) \overset{\text{(ii)}}{\longrightarrow} (1,322)$ & $2^4$
\end{tabular}
\end{table}

\subsection {Bijection between fixed points and negative crank}
For the bijection connecting (b) and (c) in Theorem \ref{big}, we use the following two lemmas.

\begin{lemma} \label{ncf}
If $\lambda \in P(n)$ has a fixed point $i$ and negative crank, then $\omega(\lambda) \ge i$.
\end{lemma}
\begin{proof}
Suppose to the contrary that $\omega(\lambda) \le i-1$.  Since $\lambda_i = i$, we have $\mu(\lambda) \ge i$ and 
\[\ck(\lambda) = \mu(\lambda) - \omega(\lambda) \ge i - (i-1) = 1,\]
a contradiction since $\lambda$ has negative crank.  Therefore $\omega(\lambda) \ge i$.
\end{proof}

\begin{lemma} \label{ncg}
If $\lambda \in P(n)$ has Durfee square size $d$, negative crank, and does not have a fixed point, then $\omega(\lambda) \ge d+1$.
\end{lemma}
\begin{proof}
Suppose to the contrary that $\omega(\lambda) \le d$.  Since $\lambda$ does not have a fixed point, we know $\lambda_d > d$ so that $\mu(\lambda) \ge d$ and
\[\ck(\lambda) = \mu(\lambda) - \omega(\lambda) \ge d - d = 0,\]
a contradiction since $\lambda$ has negative crank.  Therefore $\omega(\lambda) \ge d+1$.
\end{proof}

We now present a bijection showing
$F(n,k+1) \cong M_{<0}(n,k)$ for $k \ge 1$.

Given $\lambda \in F(n,k+1)$, let $i$ be the fixed point.  Let $\kappa$ be the partition of $n$ where $\lambda_i$ is replaced by $i$ parts 1.  Now $\omega(\kappa) \ge i$ and $\mu(\kappa) < i - 1$ so that $\ck(\kappa) \le i-1 - i = -1$.  Also, among the parts of $\lambda$ greater than one, only $\lambda_i$ is changed, so that $\mu$ has $k$ parts greater than one.  Together, then, $\kappa \in M_{<0}(n,k)$.  Note that if $\lambda_{i+1} = i$, then $\mu_i = i$ and $\mu$ also has a fixed point.  If instead $\lambda_{i+1} < i$, then $\mu$ does not have a fixed point.

There are two cases for the reverse map.  Given $\kappa \in M_{<0}(n,k)$ with a fixed point, say $i$, we know from Lemma \ref{ncf} that $\omega(\kappa) \ge i$.  Let $\lambda$ be the partition of $n$ where $i$ parts 1 are replaced by a part $i$.  This partition $\lambda$ has $\lambda_i = \kappa_i = i$ and the number of parts greater than one has increased by one, so $\lambda \in F(n,k+1)$.  Note also that $\lambda_{i+1} = i$.

Given $\kappa \in M_{<0}(n,k)$ with no fixed point and Durfee square size $d$, we know from Lemma \ref{ncg} that $\omega(\kappa) \ge d+1$.  Let $\lambda$ be the partition of $n$ where $d+1$ parts 1 are replaced by a part $d+1$ which increases the number of parts greater than one by one.  Further, since $\kappa$ does not have a fixed point, $\lambda_d = \mu_d \ge d+1$ and $\lambda_{d+1} = d+1$ (the new part), so that $\lambda$ has a fixed point $d+1$.  Therefore $\lambda \in F(n,k+1)$.  Note that $\lambda_{d+2} = \kappa_{d+1} \le d$ since $d$ is the size of the Durfee square of $\kappa$, so that the fixed point of $\lambda$ does not occur as both $\lambda_{d+1}$ and $\lambda_{d+2}$.

Considering the notes about whether the partitions have a fixed point repeated in the next part, these maps are inverses and the bijection is established.

See Table \ref{tab:bc} for an example of the bijection.

\begin{table}[t]
\renewcommand{\arraystretch}{1.1}
\caption{The bijection between $F^*(8,k+1)$ and $M_{<0}(8,k)$ where values of $k$ are separated by horizontal lines; recall the special status of $1^8$.} \label{tab:bc}
\begin{tabular}{c|c}
$F^*(8,k+1)$ & $M_{<0}(8,k)$ \\ \hline \hline
$1^8$ & $1^8$ \\ \hline
$62$ & $611$ \\
$521$ & $51^3$ \\
$4211$ & $41^4$ \\
$321^3$ & $31^5$ \\
$221^4$ & $21^6$ \\ \hline
$422$ & $4211$ \\
$3221$ & $321^3$ \\
$2^311$ & $221^4$ \\ \hline
$2^4$ & $2^311$
\end{tabular}
\end{table}

\subsection{Bijection between negative and positive crank} \label{negposcrank}

Finally, we give a bijection connecting (c) and (d) in Theorem \ref{big}.

Write $\lambda \in M_{<0}(n,k)$ as $(\lambda_1, \ldots, \lambda_k, 1^w)$ where $\lambda_k > 1$ and $w = \omega(\lambda)$.  Since $\lambda \in M_{<0}(n,k)$ and $n \ge 2$, we know $w \ge 2$.  Writing $m = \mu(\lambda)$, we have $\lambda_m > w \ge \lambda_{m+1}$ where $m = 0$ is possible, i.e., it could be that $w \ge \lambda_1$.  Associate $\lambda$ with
\[ \rho = (\lambda_1 - 1, \ldots, \lambda_m - 1, w, \lambda_{m+1}, \ldots, \lambda_k, 1^m)\]
which is clearly a partition of $n$.  Since $\lambda_m > w \ge 2$ (if $m \ne 0$), we have $\beta(\rho) = \beta(\lambda) + 1 = k+1$.  We claim that $\ck(\rho) > 0$.  Because $\ck(\lambda) < 0$, we have $w > m$.  Since $\rho_{m+1} = w > m$, we know $\mu(\rho) \ge m+1$.  Therefore
\[\ck(\rho) = \mu(\rho) - \omega(\rho) \ge m+1 - m = 1\]
and we conclude $\rho \in M_{>0}(n,k+1)$.

For the reverse map, write $\rho \in M_{>0}(n,k+1)$ as $(\rho_1, \ldots, \rho_{k+1}, 1^v)$ where $\rho_{k+1} > 1$ and $v = \omega(\rho)$.  Let $\ell = \rho_{v+1}$.  Since $\ck(\rho)>0$, we know $\ell \ge v+1$: If $\ell \le v$, then $\mu(\rho) - \omega(\rho) \le v - v = 0$. 
Associate $\rho$ with
\[ \lambda = (\rho_1 + 1, \ldots, \rho_v+1,\rho_{v+2}, \ldots, \rho_{k+1}, 1^\ell)\] 
which is clearly a partition of $n$.  Since $\ell = \rho_{v+1}$ has been replaced by $1^\ell$, we have $\beta(\lambda) = \beta(\rho) -1 = k$.  We claim that $\ck(\lambda) < 0$.  Since $\ell = \rho_{v+1}$, we have $\mu(\lambda) = v$ so that
\[\ck(\lambda) = \mu(\lambda) - \omega(\lambda) \le v - (v+1) = -1\]
and we conclude $\lambda \in M_{<0}(n,k)$.

It is straightforward to confirm that the maps are inverses, establishing the bijection.

See Table \ref{tab:cd} for an example of the bijection.

\begin{table}[t]
\renewcommand{\arraystretch}{1.1}
\caption{The bijection between $M_{<0}(8,k)$ and $M_{>0}(8,k+1)$ where values of $k$ are separated by horizontal lines.} \label{tab:cd}
\begin{tabular}{c|c}
$M_{<0}(8,k)$ & $M_{>0}(8,k+1)$ \\ \hline \hline
$1^8$ & $8$ \\ \hline
$611$ & $521$ \\
$51^3$ & $431$ \\
$41^4$ & $44$ \\
$31^5$ & $53$ \\
$21^6$ & $62$ \\ \hline
$4211$ & $3221$ \\
$321^3$ & $332$ \\
$221^4$ & $422$ \\ \hline
$2^311$ & $2^4$
\end{tabular}
\end{table}

Note that this bijection, considered over all values $k$, shows $m_{<0}(n) = m_{>0}(n)$.  However, unlike the bijection of Berkovich and Garvan \cite{bg}, it does not always send $\lambda \in M_{\ne 0}(n)$ to a partition whose crank is $-\ck(\lambda)$ (e.g., from Table \ref{tab:cd}, $\ck((5,1,1,1)) = -2$ and $\ck((4,3,1)) = 1$).

Actually, it is impossible to have a map that simultaneously has the desired effect on $\beta(\lambda)$ and negates $\ck(\lambda)$: The complete lists of partitions of 7 with cranks $\pm1$ are
\begin{gather*}
 \ck((5,1,1)) = \ck((3,2,1,1)) = -1, \beta((5,1,1)) = 1,  \beta((3,2,1,1)) = 2; \\
 \ck((4,2,1)) =  \ck((3,3,1)) = 1, \beta((4,2,1)) = \beta((3,3,1)) = 2.
 \end{gather*}

Table \ref{tab:all} shows the complete correspondences given by the bijections between the four sets for $n=9$.

\begin{table}[h]
\renewcommand{\arraystretch}{1.1}
\caption{The composed bijections between $x_e(9,k)$, $F^*(9,k+1)$, $M_{<0}(9,k)$, and $M_{>0}(9,k+1)$ where values of $k$ are separated by horizontal lines.} \label{tab:all}
\begin{tabular}{c|c|c|c}
$X_e(9,k)$ & $F^*(9,k+1)$ & $M_{<0}(9,k)$ & $M_{>0}(9,k+1)$ \\ \hline \hline
$1^9$ & $1^9$ & $1^9$ & $9$ \\ \hline
$81$ & $72$ & $711$ & $621$ \\
$711$ & $621$ & $61^3$ & $531$ \\
$61^3$ & $5211$ & $51^4$ & $441$ \\
$51^4$ & $421^3$ & $41^5$ & $54$ \\
$41^5$ & $321^4$ & $31^6$ &  $63$ \\
$31^6$ & $221^5$ & $21^7$ & $72$ \\ \hline
$531$ & $522$ & $5211$ & $4221$ \\
$441$ & $3^3$ & $331^3$ & $3^3$ \\
$4311$ & $4221$ & $421^3$ & $3321$ \\ 
$331^3$ & $32211$ & $321^4$ & $432$ \\
$321^4$ & $2^31^3$ & $221^5$ & $522$ \\ \hline
$3321$ & $32^3$ & $32211$ & $2^41$ \\
$3221^4$ & $2^41$ & $2^31^3$ & $32^3$
\end{tabular}
\end{table}

\section{Remaining open problem}
Unfortunately, our results do not seem to contribute to a combinatorial proof of $x_o(n,k) = m_{\ge 0}(n,k)$ requested by Andrews and Newman.  There are two insurmountable impediments to refining Konan's bijections 
\[X_o(n) \cong G(n) \cong M_{\le 0}(n) \cong M_{\ge 0}(n)\]
to account for the changes in the number of parts greater than one.

First, it is impossible to have a bijection between $X_o(n)$ and $G(n)$ with the desired effect on the number of parts greater than one since, for example, 
\begin{gather*}
X_o(5) = \{ (5), (3,2), (2,2,1), (2,1,1,1) \}  \text{ has $\beta$ values $1, 2, 2, 1$} \\
\text{and } G(5) = \{ (5), (4,1), (3,1,1), (2,1,1,1) \}  \text{ has $\beta$ values all $1$.}
\end{gather*}

Second, it is impossible to have a bijection between $M_{\le 0}(n,k)$ and $M_{\ge 0}(n,k+1)$ since, for example, 
\begin{gather*}
M_{\le 0}(5) = \{ (4,1), (3,1,1), (2,1,1,1), (1,1,1,1,1) \}  \text{ has $\beta$ values $1,1,1,0$} \\
\text{and } M_{\ge 0}(5) = \{ (5), (4,1), (3,2), (2,2,1) \}  \text{ has $\beta$ values $1,1,2,2$.}
\end{gather*}
In contrast to the bijection of \S \ref{negposcrank}, the issue here is the inclusion of the crank zero partition $(4,1)$.  Also, considering the case $n=26$ shows that a bijection between $M_{\le 0}(n,k)$ and $M_{\ge 0}(n,k+1)$ is not always possible even when the number of crank zero partitions is even.

We note without proof that Konan's bijection $G(n) \cong M_{\le 0}(n)$ does preserve the number of parts greater than one, i.e., is a bijection $G(n,k) \cong M_{\le 0}(n,k)$.

Satisfying Andrews and Newman's request for a bijection $X_o(n,k) \cong M_{\ge 0}(n,k)$ would seem to require a direct connection between the two sets or involve equinumerous sets of partitions other than $G(n)$ and $M_{\le 0}(n)$.


\begin{thebibliography}{99}

\bibitem{a}
G. E. Andrews.  \emph{$q$-series:\ Their Development and Application in Analysis, Number Theory, Combinatorics, Physics, and Computer Algebra}.  CBMS Regional Conference Series in Mathematics, vol. 66, American Mathematical Society, 1986.

\bibitem{ae}
G. E. Andrews, K. Eriksson. \emph{Integer Partitions}.  Cambridge University Press, 2004.

\bibitem{ag} 
G. E. Andrews, F. G. Garvan. Dyson's crank of a partition.  \emph{Bull. Amer. Math. Soc.} 18 (1988) 167--171.  

\bibitem{an19}
G. E. Andrews, D. Newman.  Partitions and the minimal excludant. \emph{Ann. Comb.} 23 (2019) 249--254.

\bibitem{an20}
G. E. Andrews, D. Newman. The minimal excludant in integer partitions. \emph{J. Integer Seq.} 23 (2020) 20.2.3.

\bibitem{an25}
G. E. Andrews, M. Newman. A refinement of the crank--mex theorem. Preprint (2025) arXiv:2508.17491.

\bibitem{bg}
A. Berkovich, F. G. Garvan, Some observations on Dyson's new symmetries of partitions, \textit{J. Combin. Theory Ser. A} 100 (2002) 61--93.

\bibitem{bk}
A. Blecher, A. Knopfmacher,  Fixed points and matching points in partitions, \emph{Ramanujan J.} 58 (2022) 23--41.

\bibitem{g}
F. Garvan, New combinatorial interpretations of Ramanujan's partition congruences mod 5, 7, and 11, \emph{Trans. Amer. Math. Soc.} 305 (1988) 47--77.

\bibitem{hs20}
B. Hopkins, J. A. Sellers. Turning the partition crank. \emph{Amer. Math. Monthly} 127 (2020) 654--657.

\bibitem{hs24}
B. Hopkins, J. A. Sellers,  On Blecher and Knopfmacher's fixed points for integer partitions, \emph{Discrete Math.} 347 (2024) 113839.

\bibitem{hsy}
B. Hopkins, J. A. Sellers, A. J, Yee, Combinatorial perspectives on the crank and mex partition statistics, \emph{Electron. J. Combin.} 29 (2022) P2.9.

\bibitem{k}
I. Konan.  A bijective proof and generalization of the non-negative crank--odd mex identity.  \emph{Electron. J. Combin.} 30 (2023) P1.41.

\end{thebibliography}
\end{document}